\documentclass[12pt]{amsart}
\usepackage{amssymb,amsmath,amsthm}
\usepackage{fancyhdr}
\usepackage{mathrsfs}
\usepackage{amsfonts}
\usepackage{amsthm}
\usepackage{amssymb}
\usepackage{latexsym}
\newtheorem{theorem}{Theorem}
\newtheorem{lemma}{Lemma}
\newtheorem{corollary}{Corollary}

\newtheorem{proposition}{Proposition}

\begin{document}

\title[Geometric sufficient conditions for compactness]{Geometric sufficient conditions for compactness of the complex Green operator}

\author{Samangi Munasinghe}
\author{Emil J. Straube}
\address{Department of Mathematics and computer Science\\
Western Kentucky University\\
Bowling Green, Kentucky, 42101}
\address{Department of Mathematics\\
Texas A\&M University\\
College Station, Texas, 77843}
\email{samangi.munasinghe@wku.edu}
\email{straube@math.tamu.edu}

\thanks{Research supported in part by NSF grant DMS 0758534. This research was begun during a visit by the first author to the Department of Mathematics at Texas A\&M University. She thanks the department for its hospitality.}

%\date{\today}

\begin{abstract}
We establish compactness estimates for $\overline{\partial}_{M}$ on a compact pseudoconvex CR-submanifold $M$ of $\mathbb{C}^{n}$ of hypersurface type that satisfies the (analogue of the) geometric sufficient conditions for compactness of the $\overline{\partial}$-Neumann operator given in \cite{Straube04, MunasingheStraube07}. These conditions are formulated in terms of certain short time flows in complex tangential directions. 

\end{abstract}

\maketitle

\section{Introduction}\label{intro}
The $\overline{\partial}$-complex on $\mathbb{C}^{n}$ induces the (extrinsic) tangential $\overline{\partial}_{M}$-complex on a CR-submanifold $M$ (\cite{Boggess91, ChenShaw01}). A CR-submanifold is of hypersurface type if the complex tangent bundle has codimension one inside the real tangent bundle. In this paper, we study compactness of the $\overline{\partial}_{M}$-complex on compact pseudoconvex CR-submanifolds of hypersurface type. These submanifolds are (possibly nonorientable) generalizations of boundaries of pseudoconvex domains to higher codimension.\footnote{When they are orientable, they are boundaries of complex varieties, but in general only in the sense of currents, see \cite{HarveyLawson75}.} Background material on CR-manifolds may be found in the books \cite{BER, Boggess91, ChenShaw01, Zampieri08}. 

While $\overline{\partial}_{M}$ has been known to have closed range in $\mathcal{L}^{2}$ for a long time when $M$ is the boundary of a smooth bounded pseudoconvex domain (\cite{Shaw85, Kohn86, BoasShaw86}), this property has been established only a few years ago in \cite{Nicoara06} for compact pseudoconvex orientable CR-submanifolds of hypersurface type of dimension at least five (see \cite{HarringtonRaich10} for a recent generalization that replaces pseudoconvexity by `weak $Y(q)$', a geometric condition that takes into account the form level). Compactness estimates for $\overline{\partial}_{M}$ when $M$ satisfies property($P$) where obtained in \cite{RaichStraube08} when $M$ is the boundary of a smooth bounded pseudoconvex domain and in \cite{Raich10, Straube10} when $M$ is a smooth compact pseudoconvex CR-submanifold, also with the restriction that the (real) dimension of $M$ be at least five. (The results in these references are actually more precise and in particular also involve the form level.)  
Here we derive compactness estimates for $\overline{\partial}_{M}$ under the assumption that $M$ satisfies the geometric sufficient conditions for compactness in the $\overline{\partial}$-Neumann problem from \cite{Straube04, MunasingheStraube07}. We refer the reader to the introductions in these two papers for a detailed discussion of the relevance of this condition \emph{vis-\`{a}-vis} the potential theoretic condition property($P$). When the dimension of $M$ equals three, we need to assume that $\overline{\partial}_{M}$ has closed range in $\mathcal{L}^{2}$ (whereas in the higher dimensional case, this property is a consequence of the estimates we prove). Alternatively, in this case, if orientability is added to the geometric sufficient conditions mentioned above, the closed range property can be deduced from these combined assumptions (see Lemma \ref{orientable}). 
%The three dimensional case requires considerably more work than the higher dimensional case; this is related to the asymmetry between the roles of $\overline{\partial}_{M}$ and $\overline{\partial}_{M}^{*}$ in either of the estimates \eqref{compest2} and \eqref{compest3} below (as compared to, say, \eqref{compest}).

We keep the (standard) notation from \cite{Straube10} (except that we switch from $\overline{\partial}_{b}$ to the more appropriate $\overline{\partial}_{M}$). Denote by $(m-1)$ the dimension, over $\mathbb{C}$, of $T^{1,0}(M)$, that is, the CR-dimension of $M$ (the choice $(m-1)$ rather than $m$ will be convenient later). Locally, we can choose a smooth purely imaginary vector field $T \in T(M)$ of unit length that is orthogonal to the complex tangent bundle $T^{\mathbb{C}}(M)$. The Levi form at $z \in M$ is the Hermitian form $L_{z}$ defined via
\begin{equation}\label{leviform}
[X, \overline{Y}] = L_{z}(X, \overline{Y})T \;\;\text{mod}\;\;T^{1,0}(M)\oplus T^{0,1}(M), \;\; X, Y \in T^{1,0}(M) \;.
\end{equation}
$M$ is pseudoconvex when every point has a neighborhood where the Levi form is (positive or negative) semidefinite; it is strictly pseudoconvex when semidefiniteness can be replaced by strict definiteness. Replacing $T$ by $-T$ if necessary, we may assume that we are in the positive case. Note that when $M$ is orientable, $T$ can be chosen to be defined globally. 

The inner product on $(0,q)$-forms in $\mathbb{C}^{n}$ induces a pointwise inner product on $(0,q)$ forms on $M$. Integrating this pointwise inner product against the induced Lebesgue measure on $M$ provides an $\mathcal{L}^{2}$ inner product on $M$:
\begin{equation}\label{innerproduct}
(u, v) = \int_{M} (u, v)_{z} d\mu_{M}(z) \;. 
\end{equation}
For $0 \leq q \leq (m-1)$ we denote by $\mathcal{L}^{2}_{(0,q)}(M)$ the completion, under the norm corresponding to \eqref{innerproduct}, of the space of smooth $(0,q)$-forms on $M$. Locally, choose smooth sections $L_{1}, \cdots, L_{m-1}$ of $T^{1,0}(M)$ that are orthonormal (and hence form a basis for $T^{1,0}(M)$) and denote by $\omega_{1}, \cdots, \omega_{m-1}$ the dual $(0,1)$-forms. Then, locally, any from $u \in \mathcal{L}^{2}_{(0,q)}(M)$ can be written in the form $u=\sum^{\prime}_{|J|=q}u_{J}\overline{\omega}^{J}$, where the prime indicates summation over strictly increasing $q$-tuples $J=(j_{1}, j_{2}, \cdots, j_{q})$ only, $\overline{\omega}^{J} = \overline{\omega_{j_{1}}}\wedge \cdots \wedge \overline{\omega_{j_{q}}}$, and the coefficients $u_{J}$ are locally square integrable.

In a local frame, as above, $\overline{\partial}_{M}$ and $\overline{\partial}_{M}^{*}$ are given as follows. If $u = \sum^{\prime}_{|J|=q}u_{J}\overline{\omega}^{J}$, then
\begin{equation}\label{d-bar-M}
\overline{\partial}_{M}u = \sideset{}{'}\sum_{|J|=q}\sum_{j=1}^{m-1}\overline{L_{j}}(u_{J})\overline{\omega_{j}}\wedge \overline{\omega}^{J} +  \text{terms of order zero} \; ,
\end{equation}
and
\begin{equation}\label{d-bar-M-*}
\overline{\partial}_{M}^{*}u = -\sum_{j=1}^{m-1}\sideset{}{'}\sum_{|K|=q-1} L_{j}u_{jK}\overline{\omega}^{K} + \text{terms of order zero} \; .
\end{equation}
Note that in the second sum in \eqref{d-bar-M}, the coefficients of $u$ are not differentiated (the terms are of order zero). The complex Laplacian $\Box_{M,q}$ is defined as  $\overline{\partial}_{M}\overline{\partial}_{M}^{*} + \overline{\partial}_{M}^{*}\overline{\partial}_{M}$, with domain consisting of those forms in $\mathcal{L}^{2}_{(0,q)}(M)$ where the compositions are defined. Alternatively, $\Box_{M,q}$ may be defined as the unique self-adjoint operator associated with the closed quadratic form $Q_{M,q}(u,u) = (\overline{\partial}_{M}u, \overline{\partial}_{M}u) + (\overline{\partial}_{M}^{*}u, \overline{\partial}_{M}^{*}u)$, with form domain equal to $dom(\overline{\partial}_{M}) \cap dom(\overline{\partial}_{M}^{*})$ (compare \cite{Davies95}, Theorem 4.4.2, \cite{ReedSimon80}, Theorem VIII.15). The complex Green operator $G_{M,q}$ is the inverse of the restriction of $\Box_{M,q}$ to the orthogonal complement of its kernel (which equals $\ker(\overline{\partial}_{q,M}) \cap \ker(\overline{\partial}_{q-1,M}^{*})$). It is bounded with respect to the $\mathcal{L}^{2}_{(0,q)}(M)$-norm precisely when the ranges of both $\overline{\partial}_{M,q}$ and $\overline{\partial}_{M,q-1}^{*}$ are closed. Indeed, the latter property is equivalent to the embedding $j: \text{dom}(\overline{\partial}_{M,q}) \cap \text{dom}(\overline{\partial}_{M,q-1}^{*}) \cap \ker(\Box_{M,q})^{\perp} \rightarrow \ker(\Box_{M,q})^{\perp}$ being continuous (\cite{Hormander65}, Theorem 1.1.2). But if $j$ is continuous, then so is the composition $j\circ j^{*}: \ker(\Box_{M,q})^{\perp} \rightarrow \ker(\Box_{M,q})^{\perp}$; however, this composition is $G_{M,q}$ (compare \cite{Straube10}, Theorem 2.9, where the corresponding fact for the $\overline{\partial}$-Neumann operator is shown). Conversely, if $G_{M,q}$ is bounded, it follows that $j$ is (same reference). We can now extend $G_{M,q}$ to all of $\mathcal{L}^{2}_{(0,q)}(M)$ by setting it zero on the kernel of $\Box_{M,q}$.

The remainder of the paper is organized as follows. In section \ref{CR>2}, we discuss our results for CR-subanifolds of hypersurface type of CR-dimension at least two (equivalently: of real dimension at least five). As usual, the case of CR-dimension one requires special attention, and we discuss this case in section \ref{CR=1}. Sections \ref{proofs} and \ref{proofs2} contain the proofs of the main results, Theorems \ref{main} and \ref{main2}, respectively.

\section{Results for CR-dimension at least two}\label{CR>2}
Let now $M$ be a smooth compact pseudoconvex CR-submanifold of $\mathbb{C}^{n}$ of hypersurface type. If $Z$ is a real vector field in $T(M)$ defined in some open subset of $M$, we denote by $\mathcal{F}_{Z}^{t}$ the flow generated by $Z$. For a set of real vector fields $T_{1}, \cdots, T_{s}$, we define $span_{\mathbb{R}}(T_{1}, \cdots, T_{s})$ to be the set of all linear combinations whose coefficients are smooth real-valued functions. For $z \in M$, $\lambda_{0}$ denotes the smallest eigenvalue of the Levi form at $z$. 

We can now state our results for CR-submanifolds of CR-dimension at least two (equivalently: of real dimension at least five). The sufficient conditions for compactness estimates for $\overline{\partial}_{M}$ are given in terms of certain complex tangential flows. For every weakly pseudoconvex point $P \in M$, there should exist an arbitrarily short complex tangential flow that takes $P$, hence a neighborhood of $P$, into the set of strictly pseudoconvex points of $M$. The size of this neighborhood needs to be controlled form below in terms of the length of the flow. Moreover, the directions of the flow must stay within the span of a set of directions whose Levi form is bounded (for some fixed constant) by a multiple of the smallest Levi eigenvalue. It will be seen in Remark 2 below that these conditions are in a sense `minimal'.
\begin{theorem}\label{main}
Let $M$ be a smooth compact pseudoconvex CR-submanifold of $\mathbb{C}^{n}$ of hypersurface type of CR-dimension $(m-1) \geq 2$. Denote by $K$ the (compact) set of weakly pseudoconvex points of $M$. Assume that there exist smooth (complex) tangential vector fields $X_{1}, \cdots, X_{s}$ of type $(1,0)$, defined on $M$ in a neighborhood of $K$, so that $L_{z}(X_{j},\overline{X_{j}}) \leq C\lambda_{0}(z)$ for some constant $C$, a sequence $\{a_{j}\}_{j=1}^{\infty}$ with $\lim_{j\rightarrow \infty}a_{j} = 0$, and constants $C_{1}, C_{2} >0$, and $C_{3}$ with $1 \leq C_{3}<1+2/(2m-1)$, so that the following holds. For every $j \in \mathbb{N}$ and $P \in K$ there is a real vector field $Z_{P,j} \in span_{\mathbb{R}}(ReX_{1}, ImX_{1}, \cdots, ReX_{s}, ImX_{s})$ of unit length, defined in some neighborhood of $P$ on $M$ with $max|div Z_{P,j}| \leq C_{1}$, such that $\mathcal{F}_{Z_{P,j}}^{a_{j}}\left (B(P,C_{2}(a_{j})^{C_{3}}) \cap K\right ) \subseteq M \setminus K$. Let $0 \leq q \leq (m-1)$. Then the following compactness estimate holds: for all $\varepsilon >0$, there is a constant $C_{\varepsilon}$ such that
\begin{equation}\label{compest}
\|u\|_{\mathcal{L}^{2}_{(0,q)}(M)} \leq \varepsilon \left (\|\overline{\partial}_{M}u\|_{\mathcal{L}^{2}_{(0,q+1)}(M)} + \|\overline{\partial}_{M}^{*}u\|_{\mathcal{L}^{2}_{(0,q-1)}(M)} \right ) + C_{\varepsilon}\|u\|_{W^{-1}_{(0,q)}(M)}
\end{equation}
for all $u \in dom(\overline{\partial}_{M}) \cap dom(\overline{\partial}_{M}^{*})$; if $q=0$ or $q=(m-1)$, we also require $u \perp \ker(\overline{\partial}_{M})$ or $\ker(\overline{\partial}_{M}^{*})$, respectively (with the usual convention that then $\overline{\partial}_{M}^{*}u$ and $\overline{\partial}_{M}u$, respectively, are zero).
\end{theorem}

\emph{Remark 1}: The compactness estimate \eqref{compest} implies the `customary' Sobolev estimates. For details, see \cite{KohnNirenberg65}, (proof of) Theorem 3, \cite{Raich10}, subsection 5.3.

%\vspace{0.1in}
A geometrically simple corollary to theorem \ref{main} is obtained using a suitable cone condition on the set $K$ of weakly pseudoconvex points of $M$ (compare \cite{Straube04}, Corollary 2, \cite{MunasingheStraube07}, Corollary 1). We say that $M \setminus K$ satisfies a complex tangential cone condition, if there is a (possibly small) open cone $C$ in $\mathbb{R}^{2n} \approx \mathbb{C}^{n}$ so that the following holds. For each $P \in K$, there exists a (real) complex tangential direction so that when $C$ is moved by a rigid motion to have vertex at $P$ and axis in the given complex tangential direction, then $C \cap M$ is contained in $M \setminus K$. With vector fields $X_{1}, \cdots, X_{s}$ as in Theorem \ref{main}, we say that $M \setminus K$ satisfies a cone condition with axis in $span (Re X_{1}, Im X_{1}, \cdots, Re X_{s}, Im X_{s})$ when at each $P \in K$, the axis of the cone can be chosen in $span (Re X_{1}(P), Im X_{1}(P), \cdots, Re X_{s}(P), Im X_{s}(P))$.

\begin{corollary}\label{cone}
Let $M$, $K$, and the vector fields $X_{1}, \cdots, X_{s}$ be given as in Theorem \ref{main}. If $M \setminus K$ satisfies a cone condition with axis in $span_{\mathbb{R}} (Re X_{1}, Im X_{1}, \cdots, Re X_{s}, Im X_{s})$, then the conclusions of Theorem \ref{main} hold.
\end{corollary}
\begin{proof}
The cone condition implies that the assumptions in Theorem \ref{main} are satisfied with $C_{3}=1$.
\end{proof}
We now discuss several classes of examples where Theorem \ref{main} applies. This discussion is analogous to the corresponding discussion in \cite{MunasingheStraube07}.

%\vspace{0.1in}
\emph{Example 1}: Assume that all the eigenvalues of the Levi form of $M$ are comparable.
The assumptions in the theorem are most transparent in this case. Any finite collection of complex tangential vector fields $X_{1}, \cdots, X_{s}$ will satisfy the condition $L_{z}(X_{j},\overline{X_{j}}) \leq \lambda_{0}(z)$. Taking a family that at each point $z$ in a neighborhood of $K$ spans $T^{1,0}(M)$ shows that this part of the assumption in Theorem \ref{main} and in Corollary \ref{cone} reduces to just `$Z_{P,j}$ complex tangential'. The class of pseudoconvex domains whose boundaries satisfy the comparable eigenvalues condition was studied in detail in \cite{Derridj78, Derridj91}; the same author later (see \cite{Derridj91}) studied hypersurfaces (of dimension at least five) whose Levi form satisfies this condition. In particular, $\mathcal{L}^{2}$-norms of derivatives of a form $u$ in complex tangential directions are controlled by the $\mathcal{L}^{2}$-norms of $\overline{\partial}_{b}u$, $\overline{\partial}_{b}^{*}u$, and $u$ (so called `maximal estimates' hold). 

%\vspace{0.1in}
\emph{Example 2}: When the CR-dimension of $M$ is one, there is only one eigenvalue of the Levi form. Accordingly, the direction of the vector fields $Z_{P,j}$ is not restricted as long as they are complex tangential. Theorem \ref{main} does not cover this situation, but we discuss the appropriate formulation in section \ref{CR=1}, see Theorem \ref{main2} and Corollary \ref{cone2}.

%\vspace{0.1in} 
The condition of comparable eigenvalues turns out to be too stringent. It excludes in particular the two classes of examples discussed below. Once the comparable eigenvalues assumption is dropped, however, some control over the direction of the fields $Z_{P,j}$ relative to the direction of the smallest eigenvalue of the Levi form is needed for the conclusion of Theorem \ref{main} to hold. For example, let $M$ be the boundary of a bounded smooth convex domain in $\mathbb{C}^{n}$, with $n \geq 3$. Assume it is strictly convex except for an analytic disc in the boundary (i.e. in $M$). Then for points $P$ on this disc, the fields $Z_{P,j}$ can be chosen in complex tangential directions transverse to the disc in such a way that the flow conditions are satisfied. Yet the complex Green operator on $(0,1)$-forms fails to be compact in this situation, by Theorem 1.5 in \cite{RaichStraube08}. 

%\vspace{0.1in}
\emph{Example 3}: Assume there exists a vector field $X_{1} \in T^{1,0}(M)$, defined near $K$, such that $X_{1}(z)$ points in the direction of the smallest eigenvalue of the Levi form at $z$. Then the assumption in Theorem \ref{main} says that the fields $Z_{P,j}(z)$ should be contained in the span (over $\mathbb{R}$) of $ReX_{1}(z)$ and $ImX_{1}(z)$.

%\vspace{0.1in}
\emph{Example 4}: Assume that at each point of $M$, the Levi form has at most one degenerate eigenvalue (i.e. the eigenvalue zero is taken with multiplicity at most one); in particular, the eigenvalues are not comparable when there are weakly pseudoconvex points and the CR-dimension of $M$ is at least two.  It is shown in \cite{MunasingheStraube07}, Example 3, that in this case, the assumptions in Theorem \ref{main} are satisfied, as follows. For each point of $P \in K$, there exists a vector field $X(z)$ on $M$, nonvanishing near $P$, such that $L(X,\overline{X}) \leq C\lambda_{0}(z)$,  (in particular, $L(X(z_{0}),\overline{X(z_{0})})=0$). Since $K$ is compact, we can choose finitely many of these fields so that at each point in a neighborhood of $K$ the span of their real and imaginary parts is nontrivial. Note that the assumption in Corollary \ref{cone} concerning the cone condition simplifies to the following: $K$ should satisfy a cone condition with axis in the null space of the Levi form (which at each point of $K$ is a two real dimensional subspace of $T(M)$). 

\emph{Remark 2}: In order to illustrate the role of the conditions in Theorem \ref{main} and to show in what sense they are `minimal', we consider pseudoconvex domains whose boundaries fall into the class discussed in Example 4 (this discussion is parallel to \cite{MunasingheStraube07},Remark 6). So assume $\Omega$ is a smooth bounded pseudoconvex domain in $\mathbb{C}^{n}$, the Levi form of whose boundary $M$ has, at each point, at most one degenerate eigenvalue. We let $q=1$ and assume that \eqref{compest} holds. This means that the complex Green operator on $(0,1)$-forms on $b\Omega$ is compact; hence, so is the $\overline{\partial}$-Neumann operator on $(0,1)$-forms on $\Omega$ (\cite{RaichStraube08}, Theorem 1.1; see also \cite{Khanh10} for results on equivalence of estimates for $\overline{\partial}_{b}$ and $\overline{\partial}$ within a general framework). Consequently, $b\Omega$ contains no analytic discs (\cite{SahutogluStraube05}, Theorem 1). Let $P \in b\Omega$ be a weakly pseudoconvex point, and let $X(z)$ be the vector field in Example 4 above. Denote by $T^{\theta}$ the (real) vector field $T^{\theta} = \cos(\theta)ReX + \sin(\theta)ImX$, and set $S_{X,P} = \{\mathcal{F}_{T^{\theta}}^{t}(P)\,|\, 0\leq \theta \leq 2\pi, \,0 \leq t \leq t_{0}\}$, for $t_{0}$ small enough. Then $S_{X,P}$ is a smooth two real dimensional submanifold of $b\Omega$. Because $b\Omega$ contains no analytic discs, $S_{X,P}$ contains points $\zeta$ arbitrarily close to $P$ such that $L(X(\zeta),\overline{X(\zeta)}) > 0$ (\cite{SahutogluStraube05}, Lemma 3). Because $0 < L(X(\zeta),\overline{X(\zeta)}) \leq C\lambda_{0}(\zeta)$, $\zeta$ is a strictly pseudoconvex point. As a result, for $a_{j}>0$ small, there exists a vector field $Z_{P,j}$ of the form $Z_{P,j}=\cos(\theta)ReX + \sin(\theta)ImX$ (in particular, $Z_{P,j} \in span_{\mathbb{R}}(ReX,ImX)$), near $P$, such that $\mathcal{F}_{Z_{P,j}}^{a_{j}}(z) \notin K$ for $z$ close to $P$. These gives the balls $B(P, \cdot)$ as in Theorem \ref{main}, \emph{except} for the lower bound on the radius. The uniform boundedness condition on the divergence of the fields is also satisfied. 

\vspace{0.1in}
We conclude this section with a brief description of the main ideas in the proof of Theorem \ref{main}. The estimates start with pseudolocal estimates near strictly pseudoconvex points. A patch $U$ near a weakly pseudoconvex point $P$, on the other hand, can flow along a vector field $Z_{P,j}$ onto a strictly pseudconvex patch $\mathcal{F}_{Z_{P,j}}^{a_{j}}(U)$. If $j$ is big enough so that $a_{j} \leq \varepsilon$, the distance flowed is dominated by $\varepsilon$ (here, we use that the fields $Z_{P,j}$ have unit length). Thus the $\mathcal{L}^{2}$-norm of a form $u$ over $U$ is controlled by the norm of $u$ over $\mathcal{F}_{Z_{P,j}}^{a_{j}}(U)$ (which is under control), plus the norm (over $U$) of $u(z) - u(\mathcal{F}_{Z_{P,j}}^{a_{j}})$. Writing this difference as an integral of the derivative $Z_{P,j}u$ and applying the Cauchy-Schwarz inequality shows that this contribution is dominated by $O(\varepsilon)$ times the $\mathcal{L}^{2}$ norm of $Z_{p,j}u$. (Very) roughly speaking, the latter is controlled by the $\mathcal{L}^{2}$-norms of $\overline{\partial}_{M}u$, $\overline{\partial}_{M}^{*}u$, and $u$, and we obtain \eqref{compest} upon summing over suitable patches. This is literally true in the comparable eigenvalues case of Example 1; in the general case, this applies only to certain microlocal portions of $u$. However, because this holds for all from levels, the estimate for the missing portion can be recovered by passing from $(0,q)$-forms to the `dual' level of $(0,m-1-q)$-forms (when $dim(M) \geq 5$; when $dim(M)=3$, the condition that $u$ be orthogonal to $\ker(\overline{\partial}_{M})$ serves as a substitute). Still, two issues arise, and they are handled by the additional assumptions in Theorem \ref{main}. First, a change of variable arises between $U$ and $\mathcal{F}_{Z_{P,j}}^{t}(U)$ for $0 \leq t \leq a_{j}$ when computing the contribution from $u(z) - u(\mathcal{F}_{Z_{P,j}}^{a_{j}})$. To control the Jacobians of these diffeomorphisms, we require a uniform bound on $div Z_{P,j}$. Second, different patches in general will overlap after flowing along the respective fields $Z_{P,j}$, and control over this overlap is needed. This is achieved by the size of the patch $B(P,C_{2}(a_{j})^{C_{3}}) \cap K$ relative to $a_{j}$.

\section{Results for CR-dimension one}\label{CR=1}

When the CR-dimension of $M$ is one, we need the closed range property of $\overline{\partial}_{M}$ already in the proof of the estimates in Theorem \ref{main2}. As of this writing, it is open whether this property always holds in the three dimensional embedded case. It is known to hold when $M$ is the boundary of a smooth bounded pseudoconvex domain in $\mathbb{C}^{2}$ (\cite{BoasShaw86, Kohn86}). More generally (in our situation), when $M$ is assumed orientable, closed range of $\overline{\partial}_{M}$ is a consequence of our geometric assumptions on the set $K$ of weakly pseudoconvex points (see Lemma \ref{orientable} below). In general, however, we have to add this requirement to the assumption (compare \cite{KohnNicoara06}, where the same issue arises). On the other hand, in the case of CR-dimension one, matters are simplified in that there is no restriction on the vector fields $Z_{P,j}$ other than that they be complex tangential (the comparability of the Levi form to the smallest eigenvalue is automatic). 

\begin{theorem}\label{main2}
Let $M$ be a smooth compact pseudoconvex CR-submanifold of $\mathbb{C}^{n}$ of hypersurface type, of CR-dimension one, and assume that $\overline{\partial}_{M}: \mathcal{L}^{2}(M) \rightarrow \mathcal{L}^{2}_{(0,1)}(M)$ has closed range. Denote by $K$ the set of weakly pseudoconvex points of $M$. Assume there are a sequence $\{a_{j}\}_{j=1}^{\infty}$ with $\lim_{j\rightarrow \infty}a_{j} = 0$, and constants $C_{1}, C_{2} >0$, and $C_{3}$ with $1 \leq C_{3}<1+2/3$, so that the following holds. For every $j \in \mathbb{N}$ and $P \in K$ there is a real complex tangential vector field $Z_{P,j}$ of unit length, defined in some neighborhood of $P$ on $M$ with $max|div Z_{P,j}| \leq C_{1}$, such that $\mathcal{F}_{Z_{P,j}}^{a_{j}}\left (B(P,C_{2}(a_{j})^{C_{3}}) \cap K\right ) \subseteq M \setminus K$. Then the following compactness estimates hold: for all $\varepsilon >0$, there is a constant $C_{\varepsilon}$ such that
\begin{equation}\label{compest2}
\|u\|_{\mathcal{L}^{2}_{(0,0)}(M)} \leq \varepsilon  \|\overline{\partial}_{M}u\|_{\mathcal{L}^{2}_{(0,1)}(M)} + C_{\varepsilon}\|u\|_{W^{-1}_{(0,0)}(M)}
\; ,\; u \in dom(\overline{\partial}_{M}) \cap ker(\overline{\partial}_{M})^{\perp} \; ;
\end{equation}
and
\begin{equation}\label{compest3}
\|u\|_{\mathcal{L}^{2}_{(0,1)}(M)} \leq \varepsilon  \|\overline{\partial}_{M}^{*}u\|_{\mathcal{L}^{2}_{(0,0)}(M)}  + C_{\varepsilon}\|u\|_{W^{-1}_{(0,1)}(M)}
\; , \; u \in dom(\overline{\partial}_{M}^{*}) \cap ker(\overline{\partial}_{M}^{*})^{\perp} \; .
\end{equation}
\end{theorem}

We also have the analogue of Corollary \ref{cone} (with the analogous proof).

\begin{corollary}\label{cone2}
Let $M$ be a smooth compact pseudoconvex CR-submanifold of $\mathbb{C}^{n}$ of hypersurface type, of CR-dimension one, and assume that $\overline{\partial}_{M}: \mathcal{L}^{2}(M) \rightarrow \mathcal{L}^{2}_{(0,1)}(M)$ has closed range. Denote by $K$ the set of weakly pseudoconvex points of $M$. If $M \setminus K$ satisfies a complex tangential cone condition, then the conclusions of Theorem \ref{main2} hold.
\end{corollary}

The closed range property of $\overline{\partial}_{M}$ can be verified when $M$ in Theorem \ref{main2} or Corollary \ref{cone2} is assumed orientable.

\begin{lemma}\label{orientable}
Let $M$ be a CR-submanifold of $\mathbb{C}^{n}$ of hypersurface type, of CR-dimension one, and make the assumptions in Theorem \ref{main2}, or in Corollary \ref{cone2}, with the exception of the closed range property of $\overline{\partial}_{M}$. Instead, assume that $M$ is orientable. Then still, $\overline{\partial}_{M}: \mathcal{L}^{2}(M) \rightarrow \mathcal{L}^{2}_{(0,1)}(M)$ must have closed range.
\end{lemma}
For emphasis, we formulate the following immediate corollary.
\begin{corollary}\label{orientable1}
Let $M$ be a CR-submanifold of $\mathbb{C}^{n}$ of hypersurface type, of CR-dimension one, and make the assumptions in Theorem \ref{main2}, or in Corollary \ref{cone2}, with the exception of the closed range property of $\overline{\partial}_{M}$. Instead, assume that $M$ is orientable. Then the conclusions of Theorem \ref{main2} still hold.
\end{corollary}

\begin{proof}[Proof of Lemma \ref{orientable}]
The assumptions in Theorem \ref{main2}, hence in Corollary \ref{cone2}, imply in particular that for every point in $M$ there is an arbitrarily short complex tangential curve connecting it to a strictly pseudoconvex point. Therefore, $M$ cannot contain analytic discs, and the closed range property of $\overline{\partial}_{M}$ follows from Proposition \ref{closed range}.
\end{proof}

\begin{proposition}\label{closed range}
Let $M \subset \mathbb{C}^{n}$ be a smooth compact orientable pseudoconvex CR-submanifold of hypersurface type, of CR-dimension $(m-1)$. Assume $M$ contains no (germs of) analytic submanifolds of dimension $(m-1)$. Then $\overline{\partial}_{M}: \mathcal{L}^{2}_{(0,q-1)}(M) \rightarrow \mathcal{L}^{2}_{(0,q)}(M)$ has closed range, $1 \leq q \leq m$.
\end{proposition}
\begin{proof}
Because $M$ is orientable, \cite{HarveyLawson75}, Theorem I applies: $M$ bounds an analytic subvariety of $\mathbb{C}^{n}$ (note that `$M$ of hypersurface type' says that `$M$ is maximally complex', in the terminology of \cite{HarveyLawson75}). The conclusion of Proposition \ref{closed range} will follow from \cite{Kohn86}, Theorems 5.2, 5.3, if we can show that $M$ actually bounds a subvariety in the $C^{\infty}$ sense. That this is the case can be seen from Remark 5.1 in \cite{Kohn86} and an observation in \cite{Straube10} (see footnote 5 there). The argument is as follows.

Locally, near $P \in M$, $M$ is a graph over a smooth pseudoconvex hypersurface in $\mathbb{C}^{m}$ (\cite{BER, Boggess91}). We may suitably rotate coordinates in such a way that the graphing function $f$ is just the inverse of the projection $\pi: \mathbb{C}^{n} \rightarrow \mathbb{C}^{m}$, $(z_{1}, \cdots, z_{n}) \rightarrow (z_{1}, \cdots, z_{m})$, and $\pi(P) = 0$:
\begin{multline}\label{graph}
f(z_{1}, \cdots, z_{m}) = (z_{1}, \cdots, z_{m}, h_{1}(z_{1}, \cdots, z_{m}), \cdots, h_{n-m}(z_{1}, \cdots, z_{m})\; , \\
(z_{1}, \cdots, z_{m}) \in \pi(M) \; ,
\end{multline}
where each $h_{j}$ is a CR-function on $\pi(M)$, and the map $f$ given by \eqref{graph} is  CR-diffeomorphism (the inverse is $\pi|_{M}$). Because $M$ does not contain germs of $(m-1)$-dimensional complex manifolds, $\pi(M)$ does not either. Therefore, Tr\'{e}preau's extension theorem for CR-functions (\cite{Trepreau86}) applies: each of the $h_{j}$'s extends in a $C^{\infty}$ way to a one sided neighborhood on the pseudoconvex side of $\pi(M)$. (Note that because $\pi(M)$ contains no $(m-1)$-dimensional analytic manifolds, the set of points where all eigenvalues of the Levi form vanish is nowhere dense on $\pi(M)$, so that its pseudoconvex side is well defined.) Now the discussion in section 10 of \cite{HarveyLawson75} applies, and Theorem 10.4 there shows that $M$ does indeed bound a subvariety of $\mathbb{C}^{n}$ in the $C^{\infty}$ sense (this subvariety then necessarily has at most finitely many isolated singularities). 
\end{proof}

\section{Proof of Theorem \ref{main}}\label{proofs}

First, we consider the case where $1 \leq q \leq (m-2)$. It suffices to establish \eqref{compest} for $u \in C^{\infty}_{(0,q)}(M)$, as this space is dense in $dom(\overline{\partial}_{M}) \cap dom(\overline{\partial}_{M}^{*})$ in the graph norm $\|u\|_{graph} = \|\overline{\partial}_{M}u\| + \|\overline{\partial}_{M}^{*}u\| + \|u\|$ (this follows from Friedrichs' Lemma, see for example Appendix D in \cite{ChenShaw01}). The implementation of the strategy outlined in section \ref{CR>2} at first proceeds in exactly the same way as the proof of the main theorem in \cite{Straube04} (or Theorem 4.30 in \cite{Straube10a}), with the simplification that there is no need to extend the vector fields $Z_{P,j}$ from $M$. These arguments (pages 705--708 in \cite{Straube04}, bottom of page 120 to page 124 
\enlargethispage*{0.25in}
in \cite{Straube10a}) give the following estimate, which is analogous to (14) in \cite{Straube04} and (2) in \cite{MunasingheStraube07}.
\begin{multline}\label{general}
\int_{M}|u|^{2} d\mu_{M} \leq 2\varepsilon\left (\|\overline{\partial}_{M}u\|^{2} + \|\overline{\partial}_{M}^{*}u\|^{2}\right ) + C_{\varepsilon}\|u\|_{-1}^{2} \\
+ C(s,C_{2})\varepsilon^{2-(C_{3}-1)(2m-1)}\sum_{k=1}^{s}\int_{M}\left (|X_{k}u|^{2} + |\overline{X_{k}}u|^{2}\right ) d\mu_{M}  \; .
\end{multline}
Here, we have chosen $j$ big enough so that $a_{j} < \varepsilon$. $C(s,C_{3})$ denotes a constant that depends only on $s$ and on $C_{3}$. The exponent $2-(C_{3}-1)(2m-1)$ is strictly positive (because $C_{3} < 1+2/(2m-1)$); it depends on the dimension $(2m-1)$ of $M$. This dimension dependence arises from the comparison of volumes argument in \cite{Straube04} (bottom of page 707 to top of page 708, page 123 in \cite{Straube10a}). Note that the respective dimensions in \cite{Straube04} (or \cite{Straube10a}) and \cite{MunasingheStraube07} are $4$ and $2n$, giving exponents of $6-4C_{3}$ and $2n+2-2nC_{3}$, respectively. A reference for the subelliptic 1/2-estimates for $\overline{\partial}_{M}$ near a strictly pseudoconvex point is \cite{Kohn85}, Theorem 2.5. The theorem is stated globally, but the proof is local and gives the pseudolocal estimates needed. Alternatively, one can combine Theorems 8.3.5 and 8.2.5 from \cite{ChenShaw01}. In Theorem 8.3.5, $M$ is assumed oriented. But the argument is local as well, and orientability is not required for this part. 

We now begin the preparations for estimating the integrals in the last term in \eqref{general}. The first part amounts to a modification of the arguments in \cite{Derridj91b}, where all complex tangential derivatives are estimated assuming comparable eigenvalues of the Levi form. To prove the compactness estimate \eqref{compest}, it suffices to consider forms $u$ with small enough support so they can be expressed in a local frame as in section \ref{intro}. So let $u=\sideset{}{'}\sum_{|J|=q}u_{J}\overline{\omega}^{J}$. 
%Then, in view of \eqref{d-bar-M} and \eqref{d-bar-M-*},
%\begin{multline}\label{energy1}
%\|\overline{\partial}_{M}u\|^{2} + \|\overline{\partial}_{M}^{*}u\|^{2} = \left \|\sum_{j=1}^{m-1}\sideset{}{'}\sum_{|J|=q} \overline{L_{j}}u_{J} \overline{\omega_{j}}\wedge \overline{\omega}^{J} \right \|^{2} + \left \|\sideset{}{'}\sum_{|K|=q-1}\sum_{k=1}^{m-1}L_{k}u_{kK}\overline{\omega}^{K} \right \|^{2} \\
%+ O\left (\|u\|\left (\|Lu\|+\|\overline{L}u\|\right ) + \|u\|^{2} \right ) \; ,
%\end{multline}
%where $\|Lu\|$ and $\|\overline{L}u\|$ have their obvious (and usual) meaning. 
The usual argument involving integration by parts and careful `accounting' of terms (see for example the proof of Theorem 8.3.5 in \cite{ChenShaw01}) then gives
\begin{multline}\label{energy2}
\|\overline{\partial}_{M}u\|^{2} + \|\overline{\partial}_{M}^{*}u\|^{2} =  \sum_{j=1}^{m-1}\sideset{}{'}\sum_{|J|=q} \left \|\overline{L_{j}}u_{J} \right \|^{2} + \sideset{}{'}\sum_{|K|=q-1}\sum_{j,k=1}^{m-1} \left ( [L_{j},\overline{L_{k}}]u_{jK}, u_{kK}\right ) \\
+ O\left (\|u\|\left (\|Lu\| + \|\overline{L}u\|\right ) + \|u\|^{2} \right ) \; .
\end{multline}
Choosing $T$ as in section \ref{intro}, setting $c_{jk} = L(L_{j},\overline{L_{k}})$, $1 \leq j,k \leq (m-1)$, and taking real parts, this becomes
\begin{multline}\label{energy3}
\|\overline{\partial}_{M}u\|^{2} + \|\overline{\partial}_{M}^{*}u\|^{2} =  \sum_{j=1}^{m-1}\sideset{}{'}\sum_{|J|=q}\left \| \overline{L_{j}}u_{J} \right \|^{2} + Re \left (\sideset{}{'}\sum_{|K|=q-1}\sum_{j,k=1}^{m-1} \left ( c_{jk}Tu_{jK}, u_{kK}\right ) \right )\\
+ O\left (\|u\|\left (\|Lu\| + \|\overline{L}u\|\right ) + \|u\|^{2} \right ) \; ,
\end{multline}
where $Re(\cdot)$ denotes the real part of a complex number. Choose a constant $A$ such that $\sum_{|J|=q}^{'}\sum_{k=1}^{s}$ $\|\overline{X_{k}}u_{J}\|^{2} \leq A\|\overline{L}u\|^{2}$. Inserting this into $A$ times \eqref{energy3} gives
\begin{multline}\label{energy6}
A\left (\|\overline{\partial}_{M}u\|^{2} + \|\overline{\partial}_{M}^{*}u\|^{2}\right ) \geq \sum_{j=1}^{s}\sideset{}{'}\sum_{|J|=q}\left \| \overline{X_{j}}u_{J} \right \|^{2} + A\,Re \left (\sideset{}{'}\sum_{|K|=q-1}\sum_{j,k=1}^{m-1} \left ( c_{jk}Tu_{jK}, u_{kK}\right ) \right )\\
+ O\left (\|u\|\left (\|Lu\| + \|\overline{L}u\|\right ) + \|u\|^{2} \right ) \; .
\end{multline}
Next, we integrate by parts
\begin{multline}\label{byparts}
\|\overline{X_{j}}u_{J}\|^{2} = \int_{M}\overline{X_{j}}u_{J}\overline{\overline{X_{j}}u_{J}} = 
-\int_{M}u_{J}\overline{X_{j}\overline{X_{j}}u_{J}} + O(\|u\|\|\overline{L}u\|) \\
= -\int_{M}u_{k}\overline{\overline{X_{j}}X_{j}u_{J}} - \int_{M}u_{J}\overline{[X_{j},\overline{X_{j}}]u_{J}} + O(\|u\|\|\overline{L}u\|) \\
= \int_{M}X_{j}u_{J}\overline{X_{j}u_{J}} - \int_{M}u_{J}\overline{L(X_{j},\overline{X_{j}})Tu_{J}} + O\left (\|u\|\left (\|\overline{L}u\|+\|Lu\|\right )\right) \; .
\end{multline}
Taking real parts in \eqref{byparts} and inserting the resulting equation into \eqref{energy6} gives
\begin{multline}\label{energy5}
A\left (\|\overline{\partial}_{M}u\|^{2} + \|\overline{\partial}_{M}^{*}u\|^{2}\right ) \geq  
\sum_{j=1}^{m-1}\sideset{}{'}\sum_{|J|=q}\left \|X_{j}u_{J} \right \|^{2}  \\
+ Re \left (\sideset{}{'}\sum_{|K|=q-1}\sum_{j,k=1}^{m-1} \left (\left( A\,c_{jk} -  \frac{1}{q}\left (\sum_{l=1}^{s}L(X_{j},\overline{X_{j}})\right )\delta_{jk}\right )Tu_{jK}, u_{kK}\right ) \right )\\
+ O\left (\|u\|\left (\|Lu\| + \|\overline{L}u\|\right ) + \|u\|^{2} \right ) \; ,
\end{multline}
where $\delta_{jk}$ is the Kronecker symbol. The factor $1/q$ arises because every term $\left (Tu_{J},u_{J}\right )$ can be written in precisely $q$ ways as $\left (Tu_{jK},u_{jK}\right )$. 
Adding \eqref{energy5} to \eqref{energy3} implies
\begin{multline}\label{energyX}
(1+A)\left (\|\overline{\partial}_{M}u\|^{2} + \|\overline{\partial}_{M}^{*}u\|^{2}\right ) \geq 
\sum_{j=1}^{m-1}\sideset{}{'}\sum_{|J|=q}\left \|X_{j}u_{J} \right \|^{2} \,+ \,
\sum_{j=1}^{m-1}\sideset{}{'}\sum_{|J|=q}\left \| \overline{L_{j}}u_{J} \right \|^{2} \\
+ Re \left (\sideset{}{'}\sum_{|K|=q-1}\sum_{j,k=1}^{m-1} \left (\left( (1+A)c_{jk} -  \frac{1}{q}\left (\sum_{l=1}^{s}L(X_{l},\overline{X_{l}})\right )\delta_{jk}\right )Tu_{jK}, u_{kK}\right ) \right )\\
+ O\left (\|u\|\|\overline{L}u\| + \|u\|^{2} \right ) \; .
\end{multline}
We have used here that the $O(\|u\|\|Lu\|)$ terms on the right hand side can be estimated by $O(\|u\|\|\overline{L}u\|+\|u\|^{2})$ (via integration by parts).

In \eqref{energyX}, the error terms can be absorbed, at the cost of adding a constant times $\|u\|^{2}$ to the left hand side. In order to estimate the integrals in the last term in \eqref{general}, we need to control the $Re(\cdot)$-terms in \eqref{energyX} from below. This control will result from an application of the sharp form of G\aa rding's inequality for vector valued (form valued, in our case) functions. Here the assumption in Theorem \ref{main} that $L_{z}(X_{j},\overline{X_{j}}) \leq C\lambda_{0}(z)$, $z \in M$, $1 \leq j \leq s$, enters. Namely, consider the eigenvalues of the matrix $\left( (1+A)c_{jk} - \frac{1}{q}\left (\sum_{l=1}^{s}L(X_{l},\overline{X_{l}})\right )\delta_{jk}\right )$. At a point $z \in M$, the sum of $q$ eigenvalues is at least $(1+A)$ times the sum of the smallest $q$ eigenvalues of $(c_{jk})$ minus $sC\lambda_{0}(z)$, so is at least $(1+A)q\lambda_{0}(z) - sC\lambda_{0}(z) = \left ((1+A)q-sC\right )\lambda_{0}(z) \geq 0$ if $A$ is chosen sufficiently large (since $\lambda_{0}(z) \geq 0$, by pseudoconvexity of $M$). In turn, this implies that
$\sideset{}{'}\sum_{|K|=q-1}\sum_{j,k=1}^{m-1} \left( (1+A)c_{jk} -  \frac{1}{q}\left (\sum_{l=1}^{s}L(X_{l},\overline{X_{l}})\right )\delta_{jk}\right )$ $v_{jK}\overline{v_{kK}}$ is  positive semidefinite on $(0,q)$-forms (see for example \cite{Straube10a}, Lemma 4.7, for a proof of this fact from multilinear algebra). It is this latter property that is needed to apply G\aa rding's inequality, as follows.

Split $u$ microlocally as in \cite{Kohn85, Kohn02}. We can take the support of $u$ small enough so that in a neighborhood $U$ of it, we have coordinates on $M$ of the form $(x_{1}, x_{2}, \cdots, x_{2m-2}, t)$ such that $T = (-i)\partial/\partial t$. Choose $\chi \in C^{\infty}_{0}(U)$ with $\chi \equiv 1$ in a neighborhood of the support of $u$. Denote the `dual' coordinates in $\mathbb{R}^{2m-1}$ by $(\xi_{1}, \cdots, \xi_{2m-2}, \tau)=(\xi, \tau)$. On the unit sphere $\{|\xi|^{2}+\tau^{2}=1\}$, choose a smooth function $g$, $0\leq g \leq 1$, supported in $\{\tau > (1/2)|\xi|\}$, and identically equal to one on $\{\tau \geq (3/4)|\xi|\}$. For $|(\xi,\tau)| \geq 3/4$ set $\chi^{+}(\xi,\tau)=g\left (\left (\xi,\tau\right )/|(\xi,\tau)|\right )$, then take a smooth continuation into $\{|(\xi,\tau)| < 3/4\}$ that vanishes when $|(\xi,\tau)| \leq 1/2$. Define $\chi^{-}$ by $\chi^{-}(\xi, \tau) = \chi^{+}(-\xi, -\tau)$. Denoting the Fourier transform on $\mathbb{R}^{2m-1}$ by $\mathcal{F}$, we define
\begin{equation}\label{split}
\mathcal{P}^{+}u = \chi\mathcal{F}^{-1}\chi^{+}\hat{u}\;,
\; \mathcal{P}^{-}u = \chi\mathcal{F}^{-1}\chi^{-}\hat{u}\;, \; \text{and} \; \mathcal{P}^{0}u = \chi\mathcal{F}^{-1}(1-\chi^{+}-\chi^{-})\hat{u} \;;
\end{equation}
where $\hat{u}=\mathcal{F}u$, and the operators act coefficientwise. The $\mathcal{P}^{j}$, $j=-,0,+$, are pseudodifferential operators of order zero. 

We apply \eqref{energyX} to $\mathcal{P}^{+}$. The term inside $Re(\cdot)$ equals
\begin{multline}\label{garding+}
\sideset{}{'}\sum_{|K|=q-1}\sum_{j,k=1}^{m-1} \left (\left( (1+A)c_{jk} -  \frac{1}{q}\left (\sum_{l=1}^{s}L(X_{l},\overline{X_{l}})\right )\delta_{jk}\right )T\chi\mathcal{F}^{-1}\chi^{+}\widehat{u_{jK}}, \chi\mathcal{F}^{-1}\chi^{+}\widehat{u_{k,K}}\right ) \\
= \sideset{}{'}\sum_{|K|=q-1}\sum_{j,k=1}^{m-1} \left (\chi^{2}\left( (1+A)c_{jk} -  \frac{1}{q}\left (\sum_{l=1}^{s}L(X_{l},\overline{X_{l}})\right )\delta_{jk}\right )T\mathcal{F}^{-1}\chi^{+}\widehat{u_{jK}}, \mathcal{F}^{-1}\chi^{+}\widehat{u_{k,K}}\right ) \\
\;\;\;\;\;\;\;\;\;\;\;\;\;\;\;\;\;\;\;\;\;\;\;\;\;\;\;\;\;\;\;\;\;\;\;\;\;\;\;\;\;\;\;\;\;\;\;\;\;\;\;\;\;\; + O(\|\mathcal{F}^{-1}\chi^{+}\hat{u}\|^{2}) \\
= \sideset{}{'}\sum_{|K|=q-1}\sum_{j,k=1}^{m-1} \left (\chi^{2}\left( (1+A)c_{jk} -  \frac{1}{q}\left (\sum_{l=1}^{s}L(X_{l},\overline{X_{l}})\right )\delta_{jk}\right )\mathcal{F}^{-1}\tau^{+}\chi^{+}\widehat{u_{jK}}, \mathcal{F}^{-1}\chi^{+}\widehat{u_{k,K}}\right ) \\
+ O(\|u\|^{2}) \; , \;\;\;\;\;\;\;\;\;\;
\end{multline}
where $\tau^{+}$ denotes a smooth function that vanishes for $\tau\leq 0$ and agrees with $\tau$ for $\tau \geq 1/2$ ($\tau$ is the Fourier variable dual to $t$). We have used that $T$ corresponds to multiplication by $\tau$ on the transform side, and that on the support of $\chi^{+}$, $\tau = \tau^{+}$. The error term in \eqref{garding+} arises from commuting $T$ with multiplication by $\chi$ to obtain the first equality in \eqref{garding+}. The matrix $\chi^{2}\left( (1+A)c_{jk} -  \frac{1}{q}\left (\sum_{l=1}^{s}L(X_{l},\overline{X_{l}})\right )\delta_{jk}\right )_{jk=1}^{m-1}$ is compactly supported in $\mathbb{R}^{2m-1}$, and the sum of any $q$ of its eigenvalues is nonnegative (from the discussion in the paragraph before the last). Therefore, we can apply the sharp form of G\aa rding's inequality (\cite{LaxNirenberg66}, either Theorem 3.1 or Theorem 3.2; see also \cite{Kohn02}, Lemma 2.5) to the form $\mathcal{F}^{-1}\chi^{+}\hat{u}$ to obtain that there is a constant $A$ such that
\begin{multline}\label{garding++}
Re\left (\sideset{}{'}\sum_{|K|=q-1}\sum_{j,k=1}^{m-1} \left (\chi^{2}\left( (1+A)c_{jk} -  \frac{1}{q}\left (\sum_{l=1}^{s}L(X_{l},\overline{X_{l}})\right )\delta_{jk}\right )\mathcal{F}^{-1}\tau^{+}\chi^{+}\widehat{u_{jK}}, \mathcal{F}^{-1}\chi^{+}\widehat{u_{k,K}}\right )\right ) \\
\geq -A\|\mathcal{F}^{-1}\chi^{+}\hat{u}\|^{2} = -A\|\chi^{+}\hat{u}\|^{2} \geq -A\|u\|^{2} \; .
\end{multline}
(The stronger conclusion in \eqref{garding++} available from Theorem 3.1 with $\|u\|_{-1/2}^{2}$ in place of $\|u\|^{2}$ is not relevant for our subsequent estimates, as there is already an $O(\|u\|^{2})$ term in \eqref{garding+}.) Inserting \eqref{garding+} and \eqref{garding++} into \eqref{energyX} and taking into account our earlier remarks about absorbing the error terms there gives
\begin{multline}\label{est+}
\|X\mathcal{P}^{+}u\|^{2} + \|\overline{L}\mathcal{P}^{+}u\|^{2} \lesssim \|\overline{\partial}_{M}\mathcal{P}^{+}u\|^{2} + \|\overline{\partial}_{M}^{*}\mathcal{P}^{+}u\|^{2} + \|\mathcal{P}^{+}u\|^{2} \\
\lesssim \|\mathcal{P}^{+}\overline{\partial}_{M}u\|^{2} + \|\mathcal{P}^{+}\overline{\partial}_{M}^{*}u\|^{2} + \|u\|^{2} \lesssim \|\overline{\partial}_{M}u\|^{2} + \|\overline{\partial}_{M}^{*}u\|^{2} + \|u\|^{2} \; ,
\end{multline}
where $X\mathcal{P}^{+}u$ has the obvious meaning. We have used here that the commutators $[\,\overline{\partial}_{M}, \mathcal{P}^{+}]$ and $[\,\overline{\partial}_{M}^{*}, \mathcal{P}^{+}]$ are operators of order zero (only the calculus for the basic symbol classes denoted by $S^{m}$ in \cite{Stein93}, chapter VI, is needed here), and also the estimate $\|\mathcal{P}^{+}u\|^{2} \lesssim \|\mathcal{F}^{-1}\chi^{+}\hat{u}\|^{2} \lesssim \|u\|^{2}$. 

To obtain the corresponding estimate for $\mathcal{P}^{0}u$, we note that on the support of $(1-\chi^{+} - \chi^{-})$, $|\tau|$ is dominated by $\sum_{j=1}^{m-1}|\sigma(\overline{L_{j}})|$, where $\sigma$ denotes the symbol (provided the coordinate neighborhood $U$ is chosen small enough, compare for example \cite{Kohn85, Kohn86}). Therefore, the last term on the right hand side of \eqref{garding+} is $O(\|\overline{L}u\|\|u\|)$. Insertion into \eqref{energyX} gives (as in \eqref{est+})
\begin{equation}\label{est0}
\|X\mathcal{P}^{0}u\|^{2} + \|\overline{L}\mathcal{P}^{0}u\|^{2} \lesssim \|\overline{\partial}_{M}u\|^{2} + \|\overline{\partial}_{M}^{*}u\|^{2} + \|u\|^{2} \; .
\end{equation}
(More is true, but not needed here: the right hand side of \eqref{est0} controls $\|\mathcal{P}^{0}u\|_{1}^{2}$, that is, $\overline{\partial}_{M} \oplus \overline{\partial}_{M}^{*}$ is elliptic on cones that avoid the $\tau$-axis; see \cite{Kohn85}, estimate (2.9), \cite{Raich10}, Lemma 4.10, \cite{Nicoara06}, Lemma 4.18.)

It remains to estimate $X\mathcal{P}^{-}u$ and $\overline{X}\mathcal{P}^{-}u$. However, these estimates come for free form those for  $X\mathcal{P}^{+}u$ and $\overline{X}\mathcal{P}^{+}u$ already established at \emph{all} form levels. More precisely, $\mathcal{P}^{-}u_{J} = \overline{\mathcal{P}^{+}\tilde{u}_{\tilde{J}}}$, up to sign (see \eqref{fourier}, \eqref{inter2}  below), where $\tilde{u}$ is the $(m-1-q)$-form `dual' to $u$ (compare for example \cite{Kohn02} for this use of the tilde-operators). $\tilde{u}$ is defined as follows. For $u=\sum^{'}u_{J}\overline{\omega}^{J}$, set $\tilde{u}:=\sum^{'}\epsilon^{J\hat{J}}_{(1,\cdots,m-1)}\overline{u_{J}}\,\overline{\omega}^{\hat{J}}$. Here, $\epsilon^{J\hat{J}}_{(1,\cdots,m-1)}$ are the generalized Kronecker symbols. Then $\|u\| = \|\tilde{u}\|$, and a short computation shows that the tilde operator intertwines $\overline{\partial}_{M}$ and $\overline{\partial}_{M}^{*}$ modulo terms of order zero (see for example \cite{Kohn02}, p.226, \cite{Koenig04}, p.289):
\begin{equation}\label{intertwine}
\overline{\partial}_{M}\tilde{u} = (-1)^{q}\widetilde{(\overline{\partial}_{M}^{*}u)} + O(\|u\|) \; , \;\;\text{and} \;\;\;
\overline{\partial}_{M}^{*}\tilde{u} = (-1)^{q+1}\widetilde{(\overline{\partial}_{M}u)} + O(\|u\|) \; .
\end{equation}
In the following computation, we denote by $h^{*}(x)$ the function $h^{*}(x):=h(-x)$ (so that $\chi^{-}=(\chi^{+})^{*}$, in this notation). Then
\begin{multline}\label{fourier}
\;\;\;\;\;\epsilon^{J\hat{J}}_{(1,\cdots,m-1)}\mathcal{P}^{-}u_{J} = \epsilon^{J\hat{J}}_{(1,\cdots,m-1)}\chi\mathcal{F}^{-1}\chi^{-}\widehat{u_{J}} = \chi\mathcal{F}^{-1}\chi^{-}\widehat{\left (\overline{\tilde{u}_{\tilde{J}}}\right )} \\
= \chi\mathcal{F}^{-1}(\chi^{+})^{*}\overline{\left(\widehat{\tilde{u}_{\tilde{J}}}\right )^{*}} =
\chi\mathcal{F}^{-1}\overline{\left (\chi^{+}\widehat{\tilde{u}_{\tilde{J}}}\right )^{*}} =
-\chi\overline{\mathcal{F}^{-1}\chi^{+}\widehat{\tilde{u}_{\tilde{J}}}} \; .\;\;\;\;\;\;\;\;\;
\end{multline}
That is,
\begin{equation}\label{inter2}
\epsilon^{J\hat{J}}_{(1,\cdots,m-1)}\mathcal{P}^{-}u_{J} = -\overline{\mathcal{P}^{+}\tilde{u}_{\tilde{J}}} \; 
\end{equation}
(equivalently: $\widetilde{\mathcal{P}^{-}u}=-\mathcal{P}^{+}\tilde{u}$). Now we can estimate $X\mathcal{P}^{-}u $ and  $\overline{X}\mathcal{P}^{-}u $. Using \eqref{inter2}, we have
\begin{multline}\label{est--}
\;\;\;\|X\mathcal{P}^{-}u\|^{2} +  \|\overline{X}\mathcal{P}^{-}u\|^{2} = \|\overline{X}\mathcal{P}^{+}\tilde{u}\|^{2} + \|X\mathcal{P}^{+}\tilde{u}\|^{2} \\
\lesssim \|\overline{\partial}_{M}\tilde{u}\|^{2} + \|\overline{\partial}_{M}^{*}\tilde{u}\|^{2}
+ \|\tilde{u}\|^{2}
\lesssim \|\overline{\partial}_{M}u\|^{2} + \|\overline{\partial}_{M}^{*}u\|^{2} + \|u\|^{2} \; ; \;\;\;\;\;\;\;
\end{multline}
the first inequality comes from \eqref{est+} applied to $\tilde{u}$, the second results from\eqref{intertwine}.

\vspace{0.15in}
\emph{Remark 3}: In contrast to the comparable eigenvalues situation studied in \cite{Derridj91b}, our assumptions do not imply any domination of (the largest eigenvalue of) $(c_{jk})$ by $\sum_{j=1}^{s-1}L(X_{j},\overline{X_{j}})$. As a result, the estimates for $X\mathcal{P}^{-}u$ and $\overline{X}\mathcal{P}^{-}u$ could not have been derived as in \cite{Derridj91b}, proceeding from a version of \eqref{energyX} where the  quadratic form in the $Re(\cdot)$-part is negative semidefinite. 

\vspace{0.15in}
We now return to \eqref{general}. Because $\|Xu\|^{2} + \|\overline{X}u\|^{2} \lesssim \sum_{j \in \{-,0,+\}}\left (\|X\mathcal{P}^{j}u\|^{2} + \|\overline{X}\mathcal{P}^{j}u\|^{2} \right )$, we obtain form \eqref{general}, together with \eqref{est+}, \eqref{est0}, and \eqref{est--}
\begin{equation}\label{final}
\int_{M}|u|^{2}d\mu_{M} \lesssim \left (\varepsilon + \varepsilon^{2-(C_{3}-1)(2m-1)}\right )\left (\|\overline{\partial}_{M}u\|^{2} + \|\overline{\partial}_{M}^{*}u\|^{2} + \|u\|^{2} \right ) + C_{\varepsilon}\|u\|_{-1}^{2}\; .
\end{equation}
Now $2-(C_{3}-1)(2m-1)>0$ by assumption, so that for $\varepsilon$ small enough the contribution on the right hand side coming from $\|u\|^{2}$ can be absorbed. The resulting estimate, after rescaling $C_{\varepsilon}$ if necessary, is the required compactness estimate \eqref{compest} in Theorem \ref{main}. This concludes the proof of Theorem \ref{main} when $1 \leq q \leq (m-2)$.

It remains to consider the cases $q=0$ and $q=(m-1)$. However, in these cases the estimates follow form those for $q=1$ and $q=(m-2)$, respectively. If $u \in \ker(\overline{\partial}_{M})^{\perp} \subset\mathcal{L}^{2}(M)$, then $u=(\overline{\partial}_{M}^{*}G_{1})\overline{\partial}_{M}u$, where $G_{1}$ is the complex Green operator on $(0,1)$-forms. By the compactness estimate \eqref{compest} for $(0,1)$-forms already established, $\overline{\partial}_{M}^{*}G_{1}$ is compact (compare Proposition 4.2 in \cite{Straube10}, where the details are given in case of the $\overline{\partial}$-Neumann operator; the arguments for the complex Green operator are the same). This gives \eqref{compest} (with $\|\overline{\partial}_{M}^{*}u\|=0$; see for example \cite{Straube10}, Lemma 4.3, for the relevant functional analysis). The argument when $q=(m-1)$ is analogous.

\section{Proof of Theorem \ref{main2}}\label{proofs2}

In this section, we show what changes/additions need to be made in the proof of Theorem \ref{main} given in the previous section in order to prove Theorem \ref{main2}. We first consider \eqref{compest2}, for $u \in \ker(\overline{\partial}_{M})^{\perp}$. Assume for the moment also that $u \in C^{\infty}(M)$. For such $u$, there is a pseudolocal $1/2$-estimate near strictly pseudoconvex points for $u$, see \cite{Kohn85}, Theorem1.3. However, this part of the argument can be rephrased without using that $u \perp \ker(\overline{\partial}_{M})$; this will be convenient for later use. For a function $f$ supported near a strictly pseudoconvex point of $M$, there is the following $1/2$-estimate (\cite{ChenShaw01}, Theorem 8.2.5):
\begin{equation}\label{1/2}
\|f\|_{1/2}^{2} \leq C\left (\|Lf\|^{2} + \|\overline{L}f\|^{2} + \|f\|^{2}\right ) \; ,
\end{equation}
where $L$ spans (locally) $T^{1,0}(M)$. Let $X_{k}$, $1 \leq k \leq s$, be a collection of complex tangential vector fields on $M$, with the following property. For each point on $M$, at least one of the $X_{k}$ satisfies $|X_{k}|>1/2$ at the point. (If the complex line bundle $T^{1,0}(M)$ admits a section that does not vanish in a neighborhood of $K$, this can be achieved by a single field.) Using \eqref{1/2} to control $u$ near strictly pseudoconvex points, but otherwise arguing as in \eqref{general}, we have 
\begin{equation}\label{general2}
\int_{M}|u|^{2} d\mu_{M} \lesssim \left(\varepsilon + C(s,C_{2})\varepsilon^{5-3C_{3}}\right)\left(\left\|Xu\right\|^{2}+\left\|\overline{X}u\right\|^{2}\right)
+ C_{\varepsilon}\|u\|_{-1}^{2} \; ,
\end{equation}
or, after rescaling $C_{\varepsilon}$ (note that $(5-3C_{3})>0$)
\begin{equation}\label{general2a}
\int_{M}|u|^{2} d\mu_{M} \leq \varepsilon \left(\left\|Xu\right\|^{2}+\left\|\overline{X}u\right\|^{2}\right)
+ C_{\varepsilon}\|u\|_{-1}^{2} \; .
\end{equation}
Here, we have set $\|Xu\| = \sum_{k=1}^{s}\|X_{k}u\|$ and $\|\overline{X}u\| = \sum_{k=1}^{s}\|\overline{X_{k}}u\|$. In contrast to the proof of Theorem \ref{main} in the previous section, we are now not assuming that $u$ has small support (because of the condition $u \perp \ker(\overline{\partial}_{M})$), but the small support assumption was not used in deriving \eqref{general}.

The analogue of \eqref{energy6} is 
\begin{equation}\label{2energy4}
A\|\overline{\partial}_{M}u\|^{2} \geq \sum_{j=1}^{s}\|\overline{X_{j}}u\|^{2} \; ,
\end{equation}
for a big enough constant $A$, since in local coordinates, $\overline{\partial}_{M}u = (Lu)\overline{\omega}$. The integration by parts in \eqref{byparts} becomes
\begin{equation}\label{byparts2}
\|\overline{X_{j}}u\|^{2} = \|X_{j}u\|^{2} - \int_{M}u\overline{L(X_{j},\overline{X_{j}})Tu} + O\left (\|u\|\left (\|Xu\| + \|\overline{X}u\| \right ) + \|u\|^{2} \right ) \; .
\end{equation}
We have used that the complex tangential terms in the commutators $[X_{j},\overline{X_{j}}]$ are controlled by $\|Xu\| + \|\overline{X}u\|$. Inserting \eqref{byparts2} into \eqref{energy6}, adding the result to \eqref{energy6}, taking real parts and absorbing terms gives the analogue of \eqref{energyX}, namely
\begin{equation}\label{2energyX2}
C\left (\|\overline{\partial}_{M}u\|^{2} + \|u\|^{2}\right ) \geq \|Xu\|^{2} + \|\overline{X}u\|^{2} -Re\left (\left  (\sum_{j=1}^{s}L(X_{j},\overline{X_{j}})\right ) Tu\,,\,u\right ) \; ,
\end{equation}
where $C$ is a suitable constant. Note that $\sum_{j=1}^{s}L(X_{j},\overline{X_{j}}) \geq 0$.

The microlocalizations used in the previous sections are only defined locally. Accordingly, we cover $M$ by open sets $U_{j}$, $1 \leq j \leq R$, so that they are defined in each $U_{j}$. Denoting these local microlocalizations by $\mathcal{P}^{-}_{j}$, $\mathcal{P}^{0}_{j}$, and $\mathcal{P}^{+}_{j}$, we define 
the global versions by
\begin{equation}\label{globalmicro}
\mathcal{P}^{*}u = \sum_{j=1}^{R}\mathcal{P}^{*}_{j}\left (\phi_{j}u\right ) \;,\; * \in \{-,0,+\} \; ,
\end{equation}
where $\{\phi_{j}\}_{j=1}^{R}$ is a partition of unity subordinate to the cover $\{U_{j}\}_{j=1}^{R}$. Note that $\mathcal{P}^{*}_{j}\left (\phi_{j}u\right )$ is compactly supported in $U_{j}$, and so is well defined on all of $M$ (with the global smoothness properties being the same as the local ones).
Using G\aa rding's inequality in the same way as in section \ref{proofs} in deriving \eqref{est+}, we obtain from \eqref{2energyX2}, applied to $\mathcal{P}^{-}_{j}(\phi_{j}u)$, $1 \leq j \leq R$, the following estimate:
\begin{multline}\label{energy10}
\|X\mathcal{P}^{-}u\|^{2} + \|\overline{X}\mathcal{P}^{-}u\|^{2} \,\lesssim \, \sum_{j=1}^{R}\sum_{k=1}^{s}\left (\|X_{k}\mathcal{P}^{-}_{j}(\phi_{j}u)\|^{2} + \|\overline{X_{k}}\mathcal{P}^{-}_{j}(\phi_{j}u)\|^{2} \right ) \\
\lesssim \sum_{j=1}^{R}\left (\|\overline{\partial}_{M}(\phi_{j}u)\|^{2} + \|\phi_{j}u\|^{2} \right ) \,\lesssim \, \|\overline{\partial}_{M}u\|^{2} + \|u\|^{2} \; , \; u \in \mathcal{L}^{2}(M) \; ;
\end{multline}
we have used (as in \eqref{est+}) that $\|\overline{\partial}_{M}\mathcal{P}^{-}_{j}(\phi_{j}u)\|^{2} \lesssim \|\mathcal{P}^{-}_{j}\overline{\partial}_{M}(\phi_{j}u)\|^{2} + O(\|\phi_{j}u\|^{2}) \lesssim \|\overline{\partial}_{M}u\|^{2}+\|u\|^{2}$. Similarly, $\mathcal{P}^{0}$ is again benign, and we have (as in section \ref{proofs})
\begin{equation}\label{est02}
\|X\mathcal{P}^{0}u\|^{2} + \|\overline{X}\mathcal{P}^{0}u\|^{2} \lesssim \|\overline{\partial}_{M}u\|^{2} + \|u\|^{2} \; , \; u \in \mathcal{L}^{2}(M) \; .
\end{equation}
With this, $\mathcal{P}^{-}u$ and $\mathcal{P}^{0}u$ are essentially under control: inserting \eqref{energy10} and \eqref{est02} into \eqref{general2a} gives the estimate
\begin{equation}\label{est-0}
\|\mathcal{P}^{-}u\|^{2} + \|\mathcal{P}^{0}u\|^{2} \leq \varepsilon\left(\|\overline{\partial}_{M}u\|^{2} + \|u\|^{2}\right) + C_{\varepsilon}\|u\|_{-1}^{2}\; .
\end{equation}
We have not used, so far, that $u \perp \ker(\overline{\partial}_{M})$. Approximating $u \in \mathcal{L}^{2}(M)$ in the graph norm of $\overline{\partial}_{M}$ (via standard mollifiers and Friedrichs' Lemma, as in section \ref{proofs}) therefore shows that \eqref{est-0} holds for $u \in \text{dom}(\overline{\partial}_{M}) \subset \mathcal{L}^{2}(M)$.

Most of the additional work in the proof of Theorem \ref{main2} is required to estimate $\mathcal{P}^{+}u$. We now use that $u \in \ker(\overline{\partial}_{M})^{\perp}$. By assumption, the range of $\overline{\partial}_{M}$ is closed. Hence so is the range of $\overline{\partial}_{M}^{*}$, and consequently, this range equals $\ker(\overline{\partial}_{M})^{\perp}$. Thus $u \in \ker(\overline{\partial}_{M})^{\perp}$ implies $u = \overline{\partial}_{M}^{*}\alpha$, with $\alpha \perp \ker(\overline{\partial}_{M}^{*})$ and $\|\alpha\| \lesssim \|u\|$. In order to exploit this, we need the analogue of \eqref{energy10} for $(0,1)$-forms. The analogue of \eqref{2energy4} is 
\begin{equation}\label{energy*}
A\left (\|\overline{\partial}_{M}^{*}\alpha\|^{2} + \|\alpha\|^{2}\right ) \geq \sum_{k=1}^{s}\|X_{k}\alpha\|^{2}\; , 
\end{equation}
which follows from (in local coordinates) $\overline{\partial}^{*}(a\overline{\omega}) = La + ga$, where $g$ is a smooth function. We also assume briefly that $\alpha$ is smooth. As a result of \eqref{energy*}, we have for $(0,1)$-forms
\begin{equation}\label{energy*2}
C\left (\|\overline{\partial}_{M}^{*}\alpha\|^{2} + \|\alpha\|^{2}\right ) \geq \|X\alpha\|^{2} + \|\overline{X}\alpha\|^{2} + Re\left (\left  (\sum_{j=1}^{s}L(X_{j},\overline{X_{j}})\right ) T\alpha\,,\,\alpha\right ) \; ,
\end{equation}
instead of \eqref{2energyX2}. The global microlocalizations $\mathcal{P}^{*}\alpha$, $ * \in \{-,0,+\}$, are defined for $(0,1)$-forms in analogy to \eqref{globalmicro}. That is, if $\alpha = \sum_{j=1}^{R}\phi_{j}\alpha = \sum_{j=1}^{R}\phi_{j}(a_{j}\overline{\omega}_{j})$, we set $\mathcal{P}^{*}\alpha = \sum_{j=1}^{R}(\mathcal{P}^{*}_{j}\phi_{j}a_{j})\overline{\omega}_{j}$. The change in sign in front of the $Re$-term in \eqref{energy*2} (compared to \eqref{2energyX2}) has the effect that for $(0,1)$-forms, the part that is under control is the positively microlocalized part $\mathcal{P}^{+}\alpha$ (rather than also $\mathcal{P}^{-}\alpha$, as in \eqref{energy10}): 
\begin{equation}\label{energy11}
\|X\mathcal{P}^{+}\alpha\|^{2} + \|\overline{X}\mathcal{P}^{+}\alpha\|^{2} \lesssim \|\overline{\partial}_{M}^{*}\alpha\|^{2} + \|\alpha\|^{2} \; , \; \alpha \in \mathcal{L}^{2}_{(0,1)}(M) \; 
\end{equation}
(which is why the argument from the previous section, passing to $\alpha = \tilde{u}$, does not help here to control $\mathcal{P}^{+}u$ for a function $u$). 

For a $(0,1)$-from $\alpha$, the analogue of \eqref{general2a} is:
\begin{equation}\label{general4}
\|\alpha\|^{2}  \lesssim \varepsilon\left (\|X\alpha\|^{2}  + \|\overline{X}\alpha\|^{2}\right ) + C_{\varepsilon}\|\alpha\|_{-1}^{2} \; .
\end{equation}
Combining \eqref{general4} (for $\mathcal{P}^{+}\alpha$) with \eqref{energy11} gives
\begin{equation}\label{est1*}
\|\mathcal{P}^{+}\alpha\|^{2}  \lesssim \varepsilon\left(  \|\overline{\partial}_{M}^{*}\alpha\|^{2} + \|\alpha\|^{2}\right)  + C_{\varepsilon}\|\alpha\|_{-1}^{2} \; ; 
\end{equation}
note that the $\mathcal{P}^{+}_{j}$, hence $\mathcal{P}^{+}$, are of order zero, so that $\|\mathcal{P}^{+}\alpha\|_{-1} \lesssim \|\alpha\|_{-1}$. \eqref{est1*} then holds for $\alpha \in \text{dom}(\overline{\partial}_{M}^{*}) \subset \mathcal{L}^{2}_{(0,1)}(M)$; this is again easily checked via mollifiers and Friedrichs' Lemma.

We can now proceed with the estimate for $\|\mathcal{P}^{+}u\|^{2}$, with $u=\overline{\partial}_{M}^{*}\alpha$. The argument is the usual one for making $\overline{\partial}_{M}u$ appear. We have
\begin{multline}\label{est20}
\|\mathcal{P}^{+}u\|^{2} = \left (\mathcal{P}^{+}\overline{\partial}_{M}^{*}\alpha, \mathcal{P}^{+}u\right ) 
= \left (\overline{\partial}_{M}^{*}\mathcal{P}^{+}\alpha, \mathcal{P}^{+}u\right ) + \left ([\mathcal{P}^{+}, \overline{\partial}_{M}^{*}]\alpha, \mathcal{P}^{+}u\right ) \\
= \left (\mathcal{P}^{+}\alpha, \overline{\partial}_{M}\mathcal{P}^{+}u\right ) + \left ([\mathcal{P}^{+}, \overline{\partial}_{M}^{*}]\alpha, \mathcal{P}^{+}u\right ) \\  = \left(\mathcal{P}^{+}\alpha, \mathcal{P}^{+}\overline{\partial}_{M}u\right ) + \left(\mathcal{P}^{+}\alpha, [\overline{\partial}_{M},\mathcal{P}^{+}]u\right) + \left ([\mathcal{P}^{+}, \overline{\partial}_{M}^{*}]\alpha, \mathcal{P}^{+}u\right ) \\
\lesssim \|\mathcal{P}^{+}\alpha\|\left (\|\overline{\partial}_{M}u\| + \|u\|\right ) + \left |\left ([\mathcal{P}^{+}, \overline{\partial}_{M}^{*}]\alpha, \mathcal{P}^{+}u\right )\right | \; .
\end{multline}
We have used here that $\|\mathcal{P}^{+}\overline{\partial}_{M}u\| \lesssim \|\overline{\partial}_{M}u\|$, and that $\|[\overline{\partial}_{M},\mathcal{P}^{+}]u\| \lesssim \|u\|$ (i.e. the commutator acts as an operator of order zero). In view of $\overline{\partial}_{M}^{*}\alpha = u$ and $\|\alpha\| \lesssim \|u\|$, inserting the estimate \eqref{est1*} into the last term in \eqref{est20} gives
\begin{equation}\label{est21}
\|\mathcal{P}^{+}u\|^{2} \,\lesssim \left(\varepsilon \|u\|^{2} + C_{\varepsilon}\|\alpha\|_{-1}^{2}\right)^{1/2}\left (\|\overline{\partial}_{M}u\| + \|u\|\right ) 
 + \left |\left ([\mathcal{P}^{+}, \overline{\partial}_{M}^{*}]\alpha, \mathcal{P}^{+}u\right )\right|\;,
\end{equation}
or, after rescaling $C_{\varepsilon}$,
\begin{equation}\label{est22}
\|\mathcal{P}^{+}u\|^{2} \leq \varepsilon\left(\|\overline{\partial}_{M}u\|^{2} + \|u\|^{2}\right) + C_{\varepsilon}\|\alpha\|_{-1}^{2} + C\left |\left ([\mathcal{P}^{+}, \overline{\partial}_{M}^{*}]\alpha, \mathcal{P}^{+}u\right )\right| \; ,
\end{equation}
with $C$ independent of $\varepsilon$.

It remains to estimate the last term in \eqref{est22}. Roughly speaking, the contribition for $j$ fixed to this term is as good as a $\mathcal{P}^{+}_{j}$-term, because its symbol is supported on the support of $\chi^{+}$. More precisely, set $\alpha = a_{j}\overline{\omega}_{j}$, $1 \leq j \leq R$, where $\overline{\omega}_{j}$ is the `basis' $(0,1)$-form on $U_{j} \subseteq M$. Then the commutator we wish to estimate is
\begin{multline}\label{commutator}
\left[\mathcal{P}^{+}, \overline{\partial}_{M}^{*}\right]\alpha = \sum_{j=1}^{R}\left(\mathcal{P}^{+}_{j}\phi_{j}(-L_{j}+g_{j})\alpha - (-L_{j}+g_{j})\mathcal{P}^{+}_{j}(\phi_{j}a_{j})\right)\overline{\omega}_{j} \\
= -\sum_{j=1}^{R}\left[\mathcal{P}^{+}_{j}, L_{j}\right](\phi_{j}a_{j})\overline{\omega}_{j}
+ \sum_{j=1}^{R}\left[\mathcal{P}^{+}_{j}, g_{j}\right](\phi_{j}a_{j})\overline{\omega}_{j}
+ \sum_{j=1}^{R}\mathcal{P}^{+}_{j}\left((-L_{j}+g_{j})\phi_{j}\right)a_{j}\overline{\omega}_{j} \; ,
\end{multline}
where $L_{j}$ spans $T^{1,0}(M)$ on $U_{j}$, and $g_{j} \in C^{\infty}(U_{j})$. Each commutator $\left[\mathcal{P}^{+}_{j}, g_{j}\right]$ in the middle term on the last line in \eqref{commutator} is of order $(-1)$, and so this term is dominated by $\|\alpha\|_{-1}$. The last term in \eqref{commutator}
can be estimated by using a version of \eqref{est1*} for each $\mathcal{P}^{+}_{j}\left((-L_{j}+g_{j})\phi_{j}\right)a_{j}$. The result is:
\begin{equation}\label{est23}
\left\|\sum_{j=1}^{R}\mathcal{P}^{+}_{j}\left((-L_{j}+g_{j})\phi_{j}\right)a_{j}\overline{\omega}_{j}
\right\|
\lesssim\varepsilon\left(\|\overline{\partial}_{M}^{*}\alpha\| + \|\alpha\|\right) + C_{\varepsilon}\|\alpha\|_{-1} 
\lesssim \varepsilon\|u\| + C_{\varepsilon}\|\alpha\|_{-1} 
\end{equation}
(recall that $\|\alpha\| \lesssim \|u\|$).

To estimate the first term in the last line of \eqref{commutator}, we analyze the part of order zero of the symbol $\sigma([\mathcal{P}^{+}_{j}, L_{j}])$ (the part of first order vanishes). It will be convenient, in the following estimates concerning $\alpha$, to first assume that $\alpha$ is smooth, and then pass to $\alpha \in \text{dom}(\overline{\partial}_{M}^{*})$ via approximation. The symbol in question is a sum of three terms, all of which are of the form $h_{1}(x,t)h_{2}(\xi,\tau)$ (see for example \cite{Stein93}, Theorem 2 in section 3, chapter VI), 
%The functions $h_{1}$ and $h_{2}$ are, up to constants, first order partial derivatives with respect to $(x,t)$ and $(\xi,\tau)$ of $\sigma(L_{j})$ and $\chi^{+}$, respectively. 
with $|h_{1}(x,t)| \leq C\chi_{1}(x,t)$, where $\chi_{1}$ is a smooth nonnegative cutoff function in $C^{\infty}(U_{j})$ that equals one in a neighborhood of the support of $\chi$ ($\chi$ is the function used in the definition \eqref{split} of the (local) microlocalizations.) The contribution of such a term to $[\mathcal{P}^{+}_{j}, L_{j}]$ is thus dominated by $\|\chi_{1}(\mathcal{F}^{-1}(h_{2}\widehat{\phi_{j}a_{j}}))\overline{\omega}_{j}\|$. Therefore, we can use \eqref{general4} and \eqref{energy*2} together with G\aa rding's inequality again as in the derivation of \eqref{est+} in section \ref{proofs} to obtain that these contributions are bounded by $\varepsilon(\|\overline{\partial}_{M}^{*}\alpha\| + \|\alpha\|) + C_{\varepsilon}\|\alpha\|_{-1} \lesssim \varepsilon\|u\| + \|\alpha\|_{-1}$. 

Combining these estimates gives
\begin{equation}\label{est26}
\left\|\left[\mathcal{P}^{+}, \overline{\partial}_{M}^{*}\right]\alpha\right\| \leq \varepsilon\|u\| + C_{\varepsilon}\|\alpha\|_{-1} \; .
\end{equation}
Inserting this last estimate into \eqref{est22} yields
\begin{equation}\label{est27}
\|\mathcal{P}^{+}u\|^{2} \leq \varepsilon \left(\|\overline{\partial}_{M}u\|^{2} + \|u\|^{2}\right) + C_{\varepsilon}\|\alpha\|_{-1}^{2} \; .
\end{equation}

We can now estimate $\|u\|$. Adding \eqref{est-0} and \eqref{est27}, absorbing terms, and rescaling $C_{\varepsilon}$ gives
\begin{equation}\label{est28}
\|u\|^{2} \leq \varepsilon\|\overline{\partial}_{M}u\|^{2} + C_{\varepsilon}\left(\|u\|_{-1}^{2} + \|\alpha\|_{-1}^{2}\right) \; .
\end{equation}
Finally, because $u \rightarrow \alpha$ is compact as a map from $\mathcal{L}^{2}(M)$ to $W^{-1}_{(0,1)}(M)$ ($\|\alpha\| \lesssim \|u\|$ and $\mathcal{L}^{2}(M) \rightarrow W^{-1}(M)$ is compact), we have $\|\alpha\|_{-1}^{2} \leq \varepsilon^{\prime}\|u\|^{2} + C_{\varepsilon^{\prime}}\|u\|_{-1}^{2}$ (see again \cite{Straube10}, Lemma 4.3 for this fact from functional analysis). Inserting this into \eqref{est28}, choosing $\varepsilon^{\prime}$ small enough, and absorbing the (small) term $C_{\varepsilon}\varepsilon^{\prime}\|u\|^{2}$ then gives the required estimate \eqref{compest2} (again upon rescaling $C_{\varepsilon}$). This completes the proof of \eqref{compest2} in Theorem \ref{main2}.

To prove \eqref{compest3}, we observe that the two estimates in Theorem \ref{main2} are actually equivalent. \eqref{compest2} says that the canonical solution operator to $\overline{\partial}_{M}$ is compact as an operator from $Im(\overline{\partial}_{M})$ to $\ker(\overline{\partial}_{M}^{*})^{\perp}$ (\cite{Straube10}, Lemma 4.3). Similarly, \eqref{compest3} says that the canonical solution operator to $\overline{\partial}_{M}^{*}$ is compact as an operator from $Im(\overline{\partial}_{M}^{*}) =  \ker(\overline{\partial}_{M})^{\perp}$ to $\ker(\overline{\partial}_{M}^{*})^{\perp} = 
Im(\overline{\partial}_{M})$. But these two operators are adjoints of each other, so that one is compact if and only if the other is. This completes the proof of Theorem \ref{main2}.

\bigskip
\providecommand{\bysame}{\leavevmode\hbox to3em{\hrulefill}\thinspace}

\end{document}